\documentclass[a4paper,10pt,reqno, english]{amsart}

\usepackage{amsmath,amssymb,amscd,amsthm,amsfonts}
\usepackage{graphicx,subfigure}
\usepackage{hyperref}
\usepackage{dsfont}
\usepackage{xcolor}
\usepackage[utf8]{inputenc}
\usepackage[nobysame, alphabetic]{amsrefs}

\newtheorem{theorem}{Theorem}[subsection]

\newtheorem{lemma}[theorem]{Lemma}

\newtheorem{corollary}[theorem]{Corollary}
\newtheorem{example}[theorem]{Example}

\newtheorem{problem}[theorem]{Problem}
\newtheorem{definition}{Definition}

\newcommand{\rr}{\mathds{R}}

\newcommand{\ff}{\mathcal{F}}
\newcommand{\ww}{\mathcal{W}}

\newcommand{\aff}[1]{\mathcal{W}_{\operatorname{aff}}(#1)}
\newcommand{\iso}[1]{\mathcal{W}_{\operatorname{iso}}(#1)}

\title{Isometric and affine copies of a set in volumetric Helly results}

\author[Messina]{John A. Messina}
\address{New York University, New York, NY 10003}
\email{jam1535@nyu.edu}

\author[Sober\'on]{Pablo Sober\'on}\address{Baruch College, City University of New York, One Bernard Baruch Way, New York, NY 10010} 
\email{pablo.soberon-bravo@baruch.cuny.edu}

\thanks{This research project was done as part of the 2020 Baruch Discrete Mathematics REU, supported by NSF awards DMS-1802059, DMS-1851420, and DMS-1953141.}

\begin{document}

\begin{abstract}
We show that for any compact convex set $K$ in $\mathbb{R}^d$ and any finite family $\mathcal{F}$ of convex sets in $\mathbb{R}^d$, if the intersection of every sufficiently small subfamily of $\mathcal{F}$ contains an isometric copy of $K$ of volume $1$, then the intersection of the whole family contains an isometric copy of $K$ scaled by a factor of $(1-\varepsilon)$, where $\varepsilon$ is positive and fixed in advance.  Unless $K$ is very similar to a disk, the shrinking factor is unavoidable.  We prove similar results for affine copies of $K$.  We show how our results imply the existence of randomized algorithms that approximate the largest copy of $K$ that fits inside a given polytope $P$ whose expected runtime is linear on the number of facets of $P$. 
\end{abstract}

\maketitle

\section{Introduction}

Helly's theorem is a central result in combinatorial geometry \cites{Radon:1921vh, Helly:1923wr}.  It says that \textit{for any finite family $\ff$ of convex sets in $\rr^d$, if every $d+1$ or fewer sets from $\ff$ have a nonempty intersection, then $\ff$ has a nonempty intersection.}  This theorem has many extensions and applications in discrete geometry, topological combinatorics, and computational geometry (see, e.g., \cites{Amenta:2017ed, Holmsen:2017uf, DeLoera:2019jb} and the references therein).

 In the quantitative versions of Helly's theorem, we aim to characterize finite families of convex sets whose intersection is quantifiably large rather than simply nonempty. B\'ar\'any, Katchalski, and Pach started this direction of research when they proved a volumetric version of Helly's theorem \cites{Barany:1982ga, Barany:1984ed}.  They showed that \textit{if the intersection of any $2d$ or fewer elements of a finite family of convex sets in $\rr^d$ has volume greater than or equal to one, the intersection of the whole family must be nonempty and have volume greater than or equal to $d^{-2d^2}$.}

The guarantee on the volume of the intersection is smaller than the bound we ask in the $2d$-tuples.  Even though this volume loss has been reduced significantly \cites{Naszodi:2016he, Brazitikos:2017ts}, it is unavoidable even if we are willing to check much larger subfamilies \cite{DeLoera:2017gt}.  A way to obtain \emph{exact} quantitative Helly-type theorems, in which no such loss is present, is to impose additional conditions.

Given a family $\mathcal{W}$ of sets in $\rr^d$, we can ask if there exists a positive integer $n$ such that \textit{for any finite family $\ff$ of convex sets in $\rr^d$, if the intersection of any $n$ or fewer of them contains a set of $\mathcal{W}$ of volume one, then $\bigcap \ff$ contains a set of $\mathcal{W}$ of volume one.}   Sarkar, Xue, and Sober\'on recently showed that such a result holds for various families $\mathcal{W}$, including the family of all axis-parallel boxes or the family of all ellipsoids \cite{Sarkar:2019tp}.  We say that $\mathcal{W}$ acts as a family of witness sets.  If such a statement holds, we say that $\mathcal{W}$ admits an exact Helly theorem for the volume.

  In this manuscript, we explore whether certain new families $\mathcal{W}$ admit an exact Helly theorem for the volume.  Except for the result for ellipsoids, the families considered in \cite{Sarkar:2019tp} fix the orientation of the sets of $\mathcal{W}$.  We are interested in families of witness sets where the orientation is not fixed.  Given a convex set $K$ in $\rr^d$, we consider the families
\begin{align*}
	\aff{K} & = \{a+\alpha AK : A \in SL(d), \alpha \in \rr, a \in \rr^d\} \qquad \mbox{and}\\ 
	\iso{K} & = \{a+\alpha AK : A \in O(d), \alpha \in \rr, a \in \rr^d \}
\end{align*}

\noindent of affine or scaled isometric copies of $K$.  For scaled isometric copies, we show that unless $K$ has a sufficiently large intersection with the boundary of the minimum volume ball containing it, $\iso{K}$ does not admit an exact Helly theorem for the volume.  For affine copies, a similar statement holds for the minimum volume enclosing ellipsoid of $K$.  For a compact set $K \subset \rr^d,$ let $B(K)$ be the minimum volume ball such that $K \subset B(K)$.  We denote by $\partial B(K)$ the boundary of $B(K)$.

\begin{theorem}\label{thm:isometric-no}
	Let $K$ be a compact set in $\rr^d$.  If there is a closed half-sphere $D \subset \partial B(K)$ such that $K \cap D$ has measure $0$ under the Haar measure of $\partial B(K)$, then $\iso{K}$ does not admit  an exact Helly theorem for the volume.
\end{theorem}

\begin{theorem}\label{thm:affine-no}
	Let $K$ be a compact set in $\rr^d$ such that $B(K)$ is also the minimum volume ellipsoid containing $K$.  If there is a closed half-sphere $D \subset \partial B(K)$ such that $K \cap D$ has measure $0$ under the Haar measure of $\partial B(K)$, then $\aff{K}$ does not admit  an exact Helly theorem for the volume.
\end{theorem}

\begin{figure}
	\centerline{\includegraphics[scale=0.3]{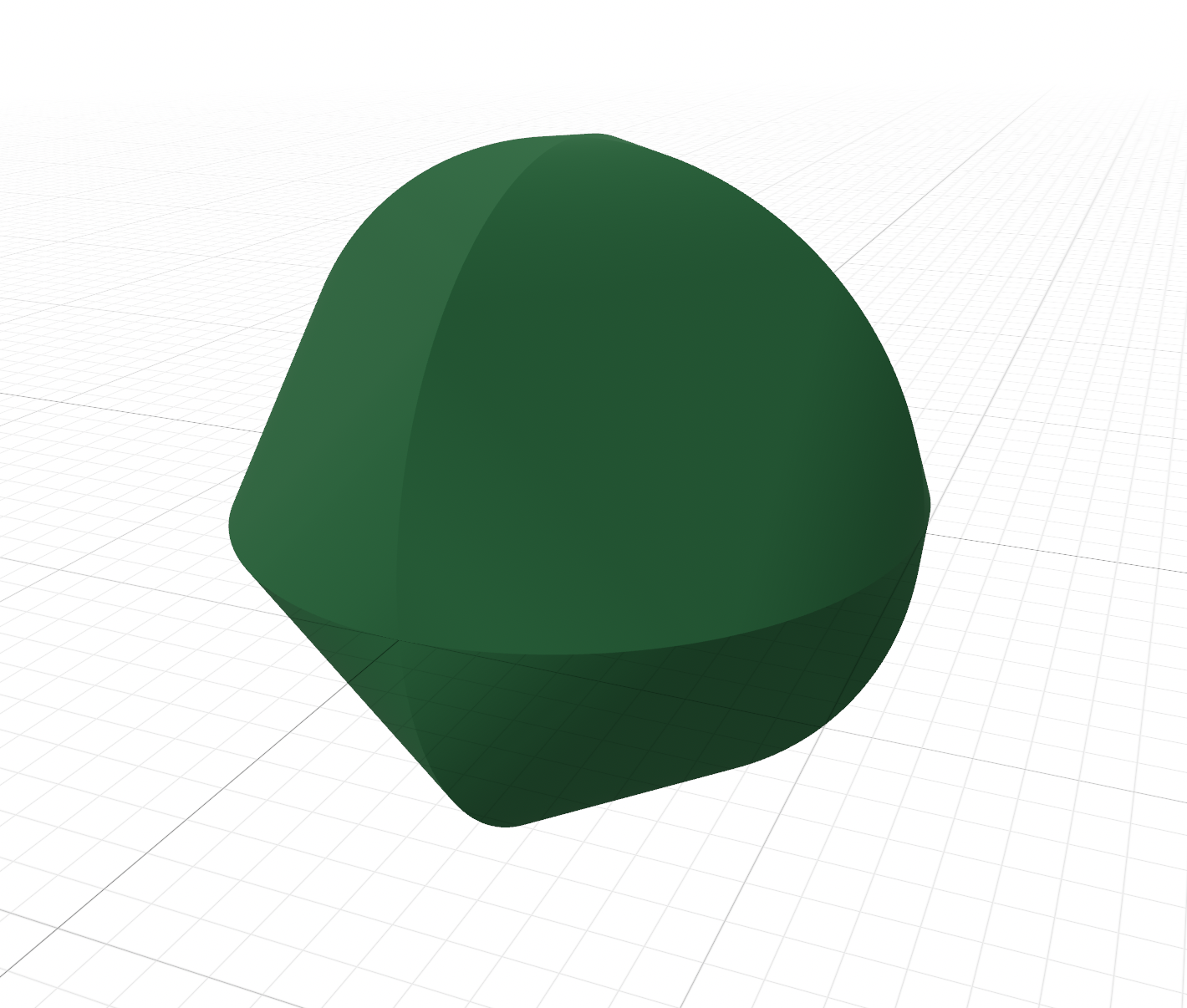}}
	\caption{Consider the three great circles in $S^2$ formed by the $xy$-plane, the $yz$-plane and the $xz$-plane.  Let $K$ be the convex hull of their union.  Theorems \ref{thm:isometric-no} and \ref{thm:affine-no} show that neither $\iso{K}$ or $\aff{K}$ admit an exact Helly theorem for the volume.}
	\label{fig:circles}
\end{figure}

For any bounded set $K$ in $\rr^d$, we may assume $B(K)$ is its minimum enclosing ellipsoid by applying a particular affine transformation, so no generality is lost with the condition of Theorem \ref{thm:affine-no}.  Theorem \ref{thm:isometric-no} implies that a  convex set $K$ must be very similar to a ball for $\iso{K}$ to admit an exact Helly theorem for the volume.  Theorem \ref{thm:affine-no} implies that a convex $K$ must be very similar to an ellipsoid for $\aff{K}$ to admit an exact Helly theorem for the volume.

  In many recent results regarding volumetric Helly-type theorems, analytic properties of ellipsoids are key ingredients of the proofs \cites{Naszodi:2016he, Brazitikos:2016ja, Brazitikos:2017ts, Brazitikos:2018uc, FernandezVidal:2020iw, Damasdi:2019vm}.  Results on the sparsification of John decompositions of the identity can be translated to Helly-type theorems.  Theorems \ref{thm:isometric-no} and \ref{thm:affine-no} show that the study of ellipsoids is much more intertwined with volumetric Helly-type theorems than previously thought.  The techniques we use to prove theorems \ref{thm:isometric-no} and \ref{thm:affine-no} involve the probabilistic method.  These methods can also be used to give a lower bound for Helly numbers for sets where those theorems fail to apply.  Theorems \ref{thm:isometric-no} and \ref{thm:affine-no} apply to any polytope $K$.  There are more general sets to which they apply, as Figure \ref{fig:circles} shows.

On a positive note, we show that if we accept a loss of $\varepsilon$ on the volume, we do have such Helly-type theorems.

\begin{theorem}\label{thm:epsilonapprox}
	Let $d$ be a positive integer and $\varepsilon >0$.   Let $K$ be a compact set with a nonempty interior.  There exists an integer $n=n(K,\varepsilon)$ such that the following statement holds.  If $\ff$ is a finite family of convex sets such that the intersection of any $n$ or fewer sets of $\ff$ contains a set of $\iso{K}$ of volume one, then $\cap \ff$ contains a set of $\iso{K}$ of volume $1-\varepsilon$.
\end{theorem}

The theorem above has consequences in computational geometry.  The problem of finding the largest copy of a polygon inside another is interesting, and many algorithms have been constructed to solve instances of this problem \cites{Amenta:1994cy, Agarwal:1998fx, Cabello:2016bi, Hallholt:2006io}.  Theorem \ref{thm:epsilonapprox} shows that the problem of finding an $\varepsilon$-approximation of the largest scaled isometric copy of a polytope $K$ inside another polytope $P$ is a linear-programming type (LP-type) problem.  Therefore, it can be solved by randomized algorithms in expected linear time in terms of the number of facets of $P$ (with hidden factors depending on $K$ and $\varepsilon$, which we assume are fixed).

  As an example, consider the problem of finding the largest volume hypercube contained inside a polytope $P \subset \rr^d$ (we consider a hypercube as an isometric scaled copy of $[0,1]^d$).  If we insist that the hypercube is axis-parallel, the results of Sarkar et a. show that this is an LP-type problem \cite{Sarkar:2019tp}.  If we allow it to have any orientation, the results here show that \textit{approximating} the largest hypercube is an LP-type problem, but there is no associated Helly theorem for an exact computation.

We set preliminaries and notation in Section \ref{sec:prelim} and prove Theorems \ref{thm:isometric-no} and \ref{thm:affine-no} in Section \ref{sec:negative-results}.  We prove Theorem \ref{thm:epsilonapprox} in Section \ref{sec:positive-results}, where we also discuss the computational applications.  Finally, some future directions of research are presented in Section \ref{sec:future-work}.

\section{Preliminaries and notation}\label{sec:prelim}

Let $B_d$ be the unit ball in $\rr^d$, and let $S^{d-1}$ be its boundary. Further, let $O(d)$ denote the group of orthogonal $d\times d$ matrices, and $SL(d)$ denote the group of $d\times d$ matrices of determinant $1.$ We denote by $\mu$ the Haar probability measure on $S^{d-1}$, which is invariant under $O(d)$.  We denote by $\langle \cdot, \cdot \rangle$ the standard dot product in $\rr^d$.

For a unit vector $u$, we say that $D(u) \subset S^{d-1}$ is the half-sphere with direction $u$ if $D(u) = \{x \in S^{d-1}: \langle x, u\rangle \ge 0\}$.  For $\varepsilon > 0$, we define the $\varepsilon$-neighborhood of $D(u)$ as the set $D(u)_{\varepsilon} = \{x \in S^{d-1}: \langle x, u\rangle > -\varepsilon \}$.

For a unit vector $u$, we denote by $H_u$ the closed half-space that contains $B_d$ and whose boundary contains $u$.  In other words
\[
H_u = \{x \in \rr^d : \langle x, u \rangle \le 1\}.
\]

For a compact set $M \subset B_d$ and $\lambda > 0$ we define a family of half-spaces
\[
\ff(M,\lambda) = \{H_u : \operatorname{dist}(u, M) \ge \lambda\}.
\]

The distance above is computed using the Euclidean distance in $\rr^d$.  Given a $d\times d$ matrix $A$ and a set $K \subset \rr^d$, we denote by $AK$ the set $\{Ax : x \in K\}$.  We can parametrize $\iso{K}$ by triples $(a,\alpha, A) \in \rr^d \times \rr \times O(d)$.  The triple $(a, \alpha, A)$ corresponds to the set $a+\alpha AK$.  A set may be represented many times if $K=AK$ for more than one matrix $A$ in $O(d)$ or $SL(d)$.  For a compact set $M \subset \rr^d$ we define
\[
S(M) = \{(a,\alpha, A)\in \rr^d \times \rr \times O(d) : a+\alpha A K \subset M\}.
\]

By checking subsequences it is simple to note that if $M$ is compact, $S(M)$ is compact.

\begin{definition}
	Let $\mathcal{W}$ be a family of sets in $\rr^d$.  We denote by $h(\mathcal{W})$  the smallest positive integer $n$, if it exists, such that the following holds.  For any finite family of convex sets in $\rr^d$, if the intersection of $n$ or fewer of them contains a set of $\mathcal{W}$, then the intersection of the whole family contains a set of $\mathcal{W}$.  If no such $n$ exists, we say $h(\mathcal{W}) = \infty$.
\end{definition}

We say that $h(\mathcal{W})$ is the Helly number for $\mathcal{W}$.  The family $\mathcal{W}$ admits an exact Helly theorem for the volume if $h(\mathcal{W}') < \infty$ for $\mathcal{W}' = \{W \in \mathcal{W}: \operatorname{vol}(W)\ge 1 \}$.

\section{Isometric and affine copies}\label{sec:negative-results}

In order to prove Theorem \ref{thm:isometric-no}, it suffices to construct for each positive integer $n$ a finite family of convex sets whose intersection does not contain an element of $\iso{K}$ of volume $1$, but the intersection of any $n$ sets does.

\begin{lemma}\label{lem:rotation}
	Let $n, d$ be positive integers.  Let $K \subset B_d$ be a compact set such that $\mu(K\cap S^{d-1}) < 1/n$.  There exists a positive constant $\delta = \delta(n,d, K)>0$ such that the following holds.  For any set $L$ of at most $n$ points in $S^{d-1}$, there exists $A \in O(d)$ such that
	\[
	\operatorname{dist}(L, AK) \ge \delta.
	\]
\end{lemma}

\begin{proof}
	Let $(x_1, \ldots, x_n) \in (S^{d-1})^n$ be an $n$-tuple of points in $S^{d-1}$.  If we pick a random matrix $A \in O(d)$, then the probability
	\[
	\mathds{P}[x_1 \in AK] = \mu (K \cap S^{d-1}) < \frac{1}{n}.
	\]
	By a simple union bound, there exists a matrix $A \in O(d)$ such that none of $x_1, \ldots, x_n$ are contained in $AK$.  Let
	\begin{align*}
		f:\left(S^{d-1}\right)^n & \to \rr \\
		(x_1, \ldots, x_n) & \mapsto \max_{A \in O(d)} \operatorname{dist}\left(\{x_1,\ldots, x_n\}, AK\right).
	\end{align*}
	We know that $f(x_1,\ldots, x_n) > 0$ for all $(x_1, \ldots, x_n) \in \left(S^{d-1}\right)^n$.  The function $f$ is also continuous.  Since the domain is compact, the function attains a minimum value $\delta$.  This is the constant we were looking for.
\end{proof}

\begin{lemma}\label{lem:almost-optimal}
	Let $K \subset B_d$ and $n$ be a positive integer.  Suppose there exists a half-sphere $D \subset S^{d-1}$ such that $\mu (K \cap D) < 1/n$.  Then, there exists a positive constant $\delta_2 = \delta_2 (n,d,K)>0$ such that for any collection $H_1, \ldots, H_n$ of $n$ half-spaces, each containing the unit ball, their intersection contains an isometric copy of $K$ scaled by a factor of $1+\delta_2$.
\end{lemma}

An intuitive illustration of the proof below is presented in Figure \ref{fig-proof}.

\begin{proof}
	Let $u$ be the unit vector such that $D=D(u)$.  We can find a positive $\varepsilon_1$ such that $\mu(K\cap D_{\varepsilon_1}(u)) < 1/n$.  Let $x_i$ be the contact point of $B_d$ and $H_i$.
	
	Consider the set
	\[
	K_{\varepsilon_1}= \{x \in K: \langle x, u\rangle \ge -\varepsilon_1 \}.
	\]
	  By Lemma \ref{lem:rotation}, we know there exists a $\delta>0$ and an isometry $T$ such that $\operatorname{dist}(\{x_1, \ldots, x_n\},TK_{\varepsilon_1})) \ge \delta$.  Therefore, for each $i$ we have
	  \[
	  H_i \in \ff(TK_{\varepsilon_1}, \delta),
	  \]
	  where $\ff (TK_{\varepsilon_1}, \delta)$ is the family defined  in Section \ref{sec:prelim}.  If we consider $Q = \bigcap \ff( K_{\varepsilon_1}, \delta)$, we have
	  \[
	  \bigcap_{i=1}^n H_i \supset TQ.
	  \]
	  Therefore, it suffices to show that $Q$ contains an isometric copy of $K$ scaled by a factor greater than one.  Let $\lambda = (1/2)\min\{\varepsilon_1, \delta \}$.  The translate $K + \lambda u$ is in the interior of $Q$.  Therefore, there exists a constant $\delta_2>0$ so that $(1+\delta_2)(K+ \lambda u) \subset Q,$ as required.
\end{proof}

\begin{figure}
	\centerline{\includegraphics[scale=0.8]{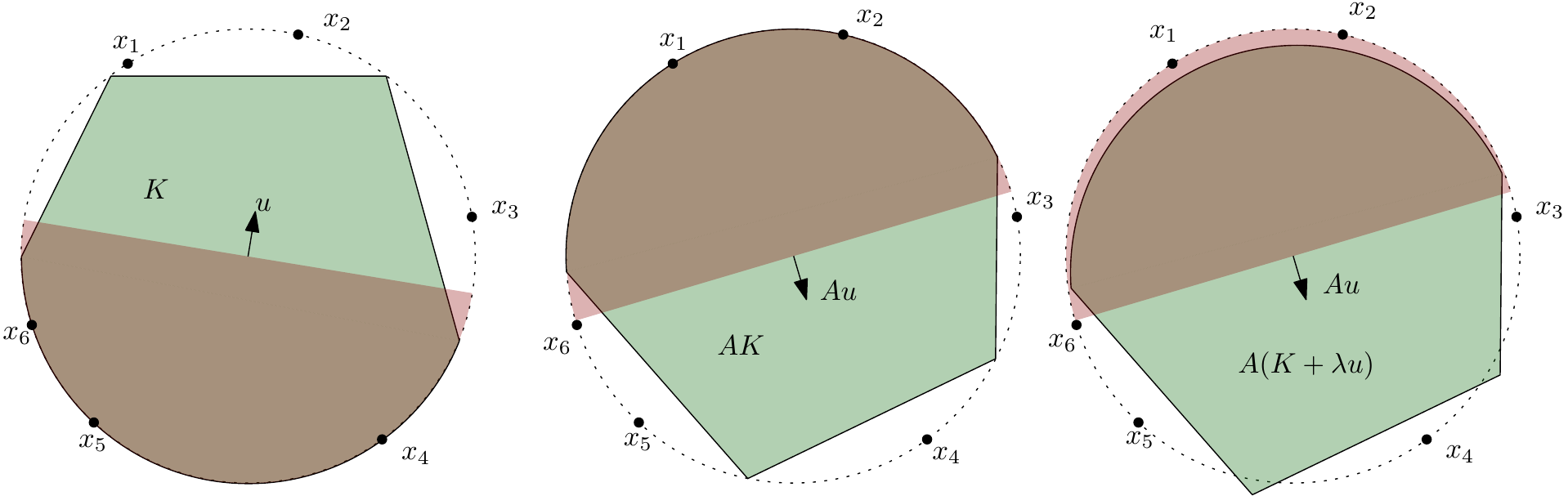}}
	\caption{We show the idea behind the proof of Lemma \ref{lem:almost-optimal}.  The part in red is the complement to $D(u)$.  We rotate $K$ so that $K \cap D$ is far from any $x_i$, and then translate it in the direction of $u$.  At that point, it's possible to scale $K$ up while remaining inside the intersection of the half-spaces tangent at $x_i$ on the sphere.}
	\label{fig-proof}
\end{figure}

\begin{proof}[Proof of Theorem \ref{thm:isometric-no}]
We may assume without loss of generality that $B(K)=B_d$.  For each $u \in S^{d-1}$, let $M_u$ be a simplex that contains $B_d$ and is tangent to $B_d$ at $u$.  Let
\[
\ff = \{M_u : u \in S^{d-1}\}.
\]
For a fixed positive integer $n$, let us use $\ff$ to show that $\iso{K}$ does not admit an exact Helly theorem for the volume.  Take $\delta_2=\delta_2(n(d+1),d,K)$ from Lemma \ref{lem:almost-optimal}.  For each $M_u \in \ff$, let 
\[
S(M_u) = \{(a,\alpha, A) \in \rr^d \times \rr \times O(d): a+\alpha AK \subset M_u, \ |\alpha| = 1+\delta_2 \}.
\]
Since $M_u$ is compact, $S(M_u)$ is compact.  The intersection of any $n$ sets of $\ff$ is a polytope of at most $n(d+1)$ facets that contains $B_d$, and therefore contains an isometric copy of $(1+\delta_2)K$.  In other words, every $n$ sets of the family $\mathcal{G} = \{S(M_u) : M_u \in \ff \}$ have a nonempty intersection.  The family $\mathcal{G}$ has an empty intersection since $\bigcap \ff = B_d$ and, by construction, $B_d$ does not contain a scaled isometric copy of $K$ with a factor greater than $1$.  Since the elements of $\mathcal{G}$ are compact, there must be a finite family $\mathcal{G}'$ whose intersection is empty.  Let $\ff' \subset \ff$ be the family that corresponds to $\mathcal{G}'$.  We know that
\begin{itemize}
	\item $\ff'$ is finite,
	\item $\bigcap \ff'$ does not contain an isometric copy of $K$ scaled by a factor of $1+\delta_2$, and
	\item the intersection of every $n$ or fewer sets of $\ff'$ contains an isometric copy of $K$ scaled by a factor of $1+\delta_2$.
\end{itemize} 
Since we can construct such a family of each positive integer $n$, the family $\iso{K}$ does not admit an exact Helly theorem for volume.
\end{proof}

\begin{proof}[Proof of Theorem \ref{thm:affine-no}]
	We may assume without loss of generality that $B_d$ is the minimum volume ellipsoid containing $K$.  Let $\ff$ be the same family of convex sets as in the proof of Theorem \ref{thm:isometric-no}.  We also denote its elements by $M_u$.  For any affine copy $K'$ of $K$ contained in $B_d$, let us look at the minimum volume ellipsoid $\mathcal{E}(K')$ containing $K'$.  If $\mathcal{E}(K') \neq B_d$, since $K' \subset B_d$, we know $\mathcal{E}(K')$ has a smaller volume than $B_d$.  Therefore, the affine function that sends $\mathcal{E}(K')$ to $B_d$ must increase volume.
	
	In other words, the largest volume that an affine copy of $K$ contained in $B_d$ can have is $\operatorname{vol}(K)$.  The affine transformation associated with an affine copy of $K$ of maximal volume in $B_d$ must preserve $B_d$, and therefore be an isometry.  Let $n$ be a fixed positive integer.  By the arguments of the proof of Theorem \ref{thm:isometric-no}, the intersection of any $n$ or fewer sets of $\ff$ contains a scaled isometric copy of $K$ by a factor of $1+\delta_2$.
	
	Consider the sets of the form
	\[
S(M_u) = \{(a,\alpha, A) \in \rr^d \times \rr \times SL(d): a+\alpha AK \subset M_u, \ |\alpha| = 1+\delta_2 \}.
\]

Even though $SL(d)$ is not compact, every set $S(M_u)$ is compact.  This follows from the fact that $M_u$ is compact, so the set of matrices $A$ in the third coordinate of a point $(a, \alpha, A)$ in $S(M_u)$ is bounded.  The same argument as before allows us to extract a finite subfamily $\ff' \subset \ff$ whose intersection does not contain an affine copy of $K$ of volume $(1+\delta_2)^d \operatorname{vol}(K)$.  However, the intersection of any $n$ sets of $\ff'$ does contain such an affine copy of $K$.  Since this can be done for any $n$, then $\aff{K}$ does not admit an exact Helly theorem for the volume.
\end{proof}

\begin{example}\label{ex:cap}
	Let $C \subset S^{d-1}$ be a circular cap of $S^{d-1}$ of measure greater than $\frac{1}{2}-\frac{1}{n}$.  Consider $K = \operatorname{conv}(S^{d-1}\setminus C)$.  The arguments in the proof of Theorem \ref{thm:isometric-no} show that, if $\iso{K}$ admits an exact Helly theorem for the volume, the Helly number must be at least $n+1$.  See Figure \ref{fig:example} for an illustration in dimension two.
\end{example}

\begin{problem}\label{prob:cap}
For $K$ as in Example \ref{ex:cap}, does $\iso{K}$ admit an exact Helly theorem for the volume?
\end{problem}

\begin{figure}
    \centering
    \includegraphics{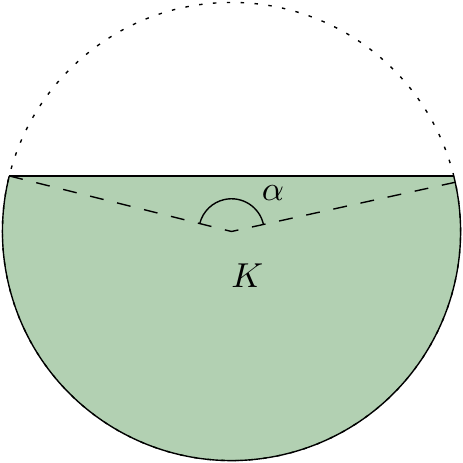}
    \caption{An illustration of Example \ref{ex:cap}.  If $\alpha>\pi-\frac{2\pi}{n}$, then $\iso{K}$ does not admit an exact Helly theorem for the volume with Helly number smaller than $n+1$.}
    \label{fig:example}
\end{figure}

\section{Approximations and computational applications}\label{sec:positive-results}

The results of this section are a consequence of the following simple lemma.

\begin{lemma}\label{lem:additive}
	Let $\mathcal{W}_1$ and $\mathcal{W}_2$ be families of sets in $\rr^d$, each of which has a finite Helly number.  Then, $\mathcal{W}_1 \cup \mathcal{W}_2$ also has a finite Helly number and
	\[
	h(\mathcal{W}_1 \cup \mathcal{W}_2) \le h(\mathcal{W}_1) + h (\mathcal{W}_2).
	\]
\end{lemma}

\begin{proof}
	We prove the contrapositive.  Let $\ff$ be a finite family of convex sets such that their intersection does not contain a set of $\mathcal{W}_1 \cup \mathcal{W}_2$.  Then, since the intersection does not contain a set of $\mathcal{W}_1$, we can find a subfamily $\ff_1 \subset \ff$ of cardinality at most $h(\mathcal{W}_1)$ whose intersection does not contain a set of $\mathcal{W}_1$.  Analogously, we can find a subfamily $\ff_2 \subset \ff$ of cardinality at most $h(\mathcal{W}_2)$ whose intersection does not contain an element of $\mathcal{W}_2$.  The family $\ff_1 \cup \ff_2$ has at most $h(\ww_1) + h(\ww_2)$ elements, and its intersection contains no element of $\ww_1 \cup \ww_2$.
\end{proof}

\begin{lemma}\label{lem:few-directions}
	Let $K$ be a compact convex set in $\rr^d$ with a nonempty interior, and $\varepsilon > 0$ be a constant.  Then, we can find a positive integer $t$ and $t$ matrices $A_1, \ldots, A_t$ in $O(d)$ such that every isometric copy of $K$ contains a translate of one of the sets $(1-\varepsilon)A_i K$.
\end{lemma}

\begin{proof}
Let $\overline{K^c}$ be the closure of the complement of $K$.  We assume without loss of generality that the origin is in the interior of $K$.  Then, $(1-\varepsilon)K$ is contained in the interior of $K$, so 
	\[
	\operatorname{dist}((1-\varepsilon)K,\overline{K^c}) > 0.
	\]
	
	Consider the function
	\begin{align*}
		f:O(d) & \to \rr \\
		A & \mapsto \max_{a \in \rr^d} \operatorname{dist}(a+A((1-\varepsilon)K), \overline{K^c}).
	\end{align*}
	
	This function is continuous.  The set $M= f^{-1}((0,\infty]) \subset O(d)$ is open and contains the identity.  For each $A \in O(d)$, consider the set 
	\[
	M_A = \{BA^{-1} : B \in M\}.
	\]
	The family $\mathcal{M} = \{M_A : A \in O(d)\}$ is an open cover of $O(d)$.  Since $O(d)$ is compact, there exists a finite collection $A_1, \ldots, A_t$ of matrices in $O(d)$ such that $M_{A_1}, \ldots, M_{A_t}$ cover $O(d)$.  Let $DK$ be an isometric copy of $K$, for some $D \in O(d)$.  Since $D^{-1} \in O(d)$, there is an $A_i$ such that $D^{-1} \in M_{A_i}$.
	
	In other words, $D^{-1}A_i \in M$.  Therefore, there exists an $a \in \rr^d$ such that
	\[
	\operatorname{dist}(a+D^{-1}A_i((1-\varepsilon)K), \overline{K^c})>0,
	\]
	which is equivalent to
	\[
	\operatorname{dist}(Da+A_i((1-\varepsilon)K), D\overline{K^c})>0.
	\]
	Finally, this means that there is a translate of $A_i((1-\varepsilon)K)$ contained in the interior of $DK$.
	\end{proof}
	
Given two compact convex sets $K$ and $P$ in $\rr^d$, an interesting problem is to find the largest scaled isometric copy of $K$ contained in $P$.  If $K$ is fixed and $P$ is a polytope with $n$ facets, we would like to know the complexity of solving this problem in terms of $n$.  Formally, we want to compute the constant
\[
\alpha(K,P) = \max \{\alpha : a + \alpha AK \subset P \mbox{ for some }a \in \rr^d, A \in O(d)\}
\]
We show how to use Lemma \ref{lem:few-directions} to find an approximation of this parameter.  First, if we are given a particular set $A_1, \ldots, A_t$ of matrices in $O(d)$ we can define a similar parameter
\[
\beta(K,P) = \max \{\alpha : a + \alpha A_iK \subset P \mbox{ for some }a \in \rr^d, 1 \le i \le t\}.
\]

\begin{theorem}
	Let $K$ be a compact convex set in $\rr^d$ with a nonempty interior. Let $A_1, \ldots, A_t$ be matrices in $O(d)$ used to define the parameter $\beta(K, \cdot )$.  For a finite family $\ff$ of convex sets in $\rr^d$, there exists a subfamily $\ff'$ of cardinality $t(d+1)$ such that
	\[
	\beta\left(K,\bigcap \ff\right) = \beta \left(K, \bigcap \ff'\right).
	\]
\end{theorem}

\begin{proof}
	We use an argument similar to the one in the proof of Lemma \ref{lem:additive}.  For each $i=1,\ldots, t$, define the parameter $\beta_i(K, \cdot)$ as
	\[
	\beta_i(K,P) = \max \{\alpha : a + \alpha A_iK \subset P \mbox{ for some }a \in \rr^d\}.
	\]
	Let $\beta_i = \beta_i(K, \bigcap \ff)$.  By Helly's theorem for translates of a set, we know that the family of witness sets
	\[
	\ww_i = \{a+\beta_i A_i K: a \in \rr^d\}
	\]
	has Helly number at most $d+1$.  Therefore, there is a family $\ff_i \subset \ff$ of size $d+1$ such that $\beta_i( K, \bigcap \ff_i) = \beta_i$.  Since $\beta(K, \cdot) = \max_{1\le i \le t} \beta_i(K, \cdot)$, it suffices to take $\ff' = \bigcup_{i=1}^t \ff_i$ to finish the proof.
\end{proof}

The theorem above shows that computing $\beta (K,P)$ is an LP-type problem.  Consider $d, t$ to be fixed.  To compute $\beta(K,P)$ for a polytope $P$, we first write $P$ as the intersection of $n$ half-spaces, $P = \cap_{i=1}^n H_i$.  A brute-force algorithm would be as follows.  For a fixed $t(d+1)$-tuple $I \subset [n]$, let $P' = \cap_{i\in I} H_i$.  We compute $\beta(K,P')$.  We repeat this for all $\binom{n}{t(d+1)} = O(n^{t(d+1)})$ different $t(d+1)$-tuples of half-spaces, and output the minimum number found.

  We can do better by applying a randomized algorithm, such as the randomized dual-simplex algorithm \cite{Sharir:1992ih} that runs in $O(n)$ expected time.  The parameters $t, d$ affects the hidden constant factor, but not the dependence on $n$.  The algorithms depend on access to an oracle that finds $\beta( K, P')$ when $P'$ is the intersection of $t(d+1)$ half-spaces.  We discuss below why such a computation is possible when $K$ is a polytope.  First, let us show how the computation of $\beta(K,P)$ implies an approximation of $\alpha(K,P)$.

\begin{corollary}
	Let $K$ be a convex polytope in $\rr^d$ whose interior is not empty, and $\varepsilon >0$ be fixed.  There is a randomized algorithm that runs in $O(n)$ expected time such that approximates $\alpha(K,P)$ up to a relative error of $\varepsilon$.
\end{corollary}

\begin{proof}
	Let $A_1, \ldots, A_t$ be the matrices from Lemma \ref{lem:few-directions}.  Then, for any polytope $P$ we have
	\[
	(1-\varepsilon)\alpha(K,P) \le \beta(K,P) \le \alpha(K,P).
	\]
	In other words, $\beta(K,\cdot)$ approximates $\alpha(K,\cdot)$ with a relative error not greater than $\varepsilon$.  We can run the randomized dual-simplex algorithm and find $\beta(K,P)$ in expected $O(n)$ time.
\end{proof}

If $K$ and $P$ are polytopes, we can check if $a+\beta_i K \subset P$ by checking the vertices of $K$ one by one.  A maximal translate of $K$ in $P$ will have contact points with facets of $P$ whose normal vectors capture the origin.

Let us look at the example of approximating the size of the largest equilateral triangle inside a polytope.  The first task, finding the value of $t$, can be done by finding the angle $\alpha$ at which any rotation of an equilateral triangle of side $1-\varepsilon$ fits inside an equilateral triangle of side $1$ (see Figure \ref{fig-triangle}).  We set $t=\pi/\alpha$. The rotations $A_1, \ldots, A_t$ are simply rotations by an angle of $\frac{2\pi j}{t}$ for $j=1,\ldots, t$.  Once $A_1, \ldots, A_t$ are fixed, we follow the algorithms described above.  In the plane, a similar process can be done for any convex polygon $K$.

\begin{figure}
	\centerline{\includegraphics{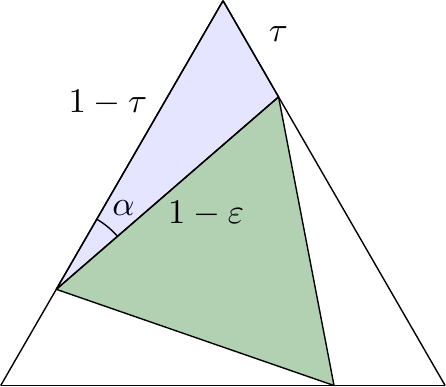}}
	\caption{To find the largest angle $\alpha$ for which a rotated copy of an equilateral triangle of side $1-\varepsilon$ fits inside a side $1$ equilateral triangle, it suffices to use the law of sines twice in the blue triangle and solve for $\alpha$.}
	\label{fig-triangle}
\end{figure}

In high dimensions, the problem of computing $t$ and the matrices $A_1, \ldots, A_t$ is interesting.  If we consider $O(d)$ as a metric space, a sufficiently dense net depending on $\varepsilon$ and $K$ will work.  The precise value of $t$ would not affect the expected time in terms of $n$ for the algorithms mentioned above.  However, those algorithms carry hidden factors in terms of the combinatorial complexity of the LP-type problem, which is $t(d+1)$.  The following problem is relevant.

\begin{problem}
	Given a polytope $K$ in $\rr^d$ and $\varepsilon >0$, compute the smallest value $t$ such that there exist $A_1, \ldots, A_t \in O(d)$ for which any isometric copy of $K$ contains a translate of $A_i((1-\varepsilon)K)$ for some $1\le i \le t$.
\end{problem}

\section{Future Directions of Research}\label{sec:future-work}

In this work, we address the problem of finding for which sets $K$ the collections $\mathcal{W}_{\textrm{iso}}$ and $\mathcal{W}_{\textrm{aff}}$ admit an exact Helly theorem for the volume. Theorem \ref{thm:isometric-no} shows that if, for a given convex set $K,$ $\mathcal{W}_{\textrm{iso}}$ admits an exact Helly theorem for the volume, then $K$ must have a large intersection with its minimal enclosing sphere. If a negative answer to Problem \ref{prob:cap} holds, one may ask the following questions.
\begin{problem}
Is $B_d$ the only set for which $\mathcal{W}_{\textrm{iso}}$ admits an exact Helly theorem for the volume?
\end{problem}
\begin{problem} Are ellipsoids the only sets for which $\mathcal{W}_{\textrm{aff}}$ admits an exact Helly theorem for the volume?
\end{problem}
One may alternatively ask which collections of copies of a given set $K$ admit an exact Helly theorem for the volume. In particular, for a set $K\subseteq \rr^d $ and a subgroup $G < O(d)$, one may ask whether the set 
\[ \mathcal{W}_G = \{a + \alpha AK | A \in G, \alpha \in \rr, a\in \rr^d\}\]
admits an exact Helly theorem for the volume. Theorem \ref{thm:isometric-no} shows that if $G = O(d),$ a negative answer holds unless $K$ is very similar to a ball. Lemma \ref{lem:additive} implies that for any finite subgroup $G$ and any compact set $K$ of positive volume, the set $\mathcal{W}_G$ admits an exact Helly theorem for the volume. For infinite subgroups, the following problem remains open.
\begin{problem}
Given a subgroup $G < O(d),$ for which convex sets $K\subset \rr^d$ does $\mathcal{W}_G$ admit an exact Helly theorem for the volume?
\end{problem}

In particular, does $K$ have to be similar to a $G$-invariant subset of $\rr^d$ as in Theorem \ref{thm:isometric-no}?

\section{Acknowledgments}

The authors thank Edgardo Rold\'an-Pensado for making Figure \ref{fig:circles}.



\end{document}